\documentclass[12pt,reqno]{article}
\usepackage{authblk}
\usepackage{amscd,amsmath,amstext,amsfonts,amsbsy,
	amssymb,amsthm,mathrsfs,float,latexsym}
\usepackage{multicol,graphicx,array,multirow,color,lineno}
\usepackage{sidecap,url,fancybox,fancyhdr}
\usepackage[colorlinks=true,linkcolor=blue,citecolor=magenta]{hyperref}

\usepackage[margin=1in]{geometry}
\renewcommand{\epsilon}{\varepsilon}
\newcommand{\eps}{\epsilon}

\newcommand{\nn}{\nonumber}
\newcommand{\MR}{\mathsf{MR}}
\newcommand{\bra}[1]{\left(#1\right)}

\newcommand{\na}{\nabla}

\newcommand{\pa}{\partial}

\newtheorem{theorem}{{\bf Theorem}}[section]
\theoremstyle{definition} \newtheorem{definition}{\bf Definition}

\theoremstyle{plain} 
\newtheorem{lemma}{Lemma}

\newtheorem{remark}{Remark}[section]

\DeclareMathOperator*{\esssup}{ess\,sup}

\usepackage[new]{old-arrows} 

\usepackage{chemfig,chemmacros}

\title{\bf \Large Fast-reaction limits for predator--prey reaction--diffusion systems: improved convergence}

\author[$\dagger$]{Cinzia Soresina}
\author[$\dagger$]{Bao Quoc Tang\thanks{Corresponding author.\\
		Email addresses: cinzia.soresina@uni-graz.at, quoc.tang@uni-graz.at, bao-ngoc.tran@uni-graz.at}}
\author[$\dagger$,$\ddagger$]{Bao-Ngoc Tran}
\affil[$\dagger$]{\small{\textit{Institute for Mathematics and Scientific Computing, University of Graz, Graz, Austria}}}
\affil[$\ddagger$]{\small{\textit{Department of Mathematics,  Nong Lam University, Ho Chi Minh City, Vietnam}}}
\date{}  

\begin{document}
	
	\maketitle

	\begin{abstract} 
	\noindent
The fast-reaction limit for reaction--diffusion systems modelling predator--prey interactions is investigated.  In the considered model,  predators exist in two possible states, namely searching and handling. The switching rate between these two states happens on a much faster time scale than other processes, leading to the consideration of the fast-reaction limit for the corresponding systems. The rigorous convergence of the solution to the fast-reaction system to the ones of the limiting cross-diffusion system has been recently studied in [Conforto, Desvillettes, Soresina, NoDEA, 25(3):24, 2018]. In this paper, we extend these results by proving improved convergence of solutions and slow manifolds. In particular, we prove that the slow manifold converges strongly in all dimensions without additional assumptions, thanks to use of a modified energy function. This consists in a unified approach since it is applicable to both types of fast-reaction systems, namely with the Lotka--Volterra and the Holling-type II terms.
		
		\vspace*{0.15cm} 
		
		
	\end{abstract}

	\section{Introduction} 
	%\textcolor{blue}{@Cinzia: Discussion about prey-predator models related to \eqref{e1} and fast reaction limit problems in the literature.}\\
	
	% Other fast reaction limit in literature.
	An interesting modelling issue in several contexts is the derivation of cross-diffusion terms via Quasi-Steady-State Approximation (QSSA) from a fast-reaction model involving different time scales and standard diffusion.  In fact,  the justification of cross-diffusion terms from semilinear reaction–didffusion systems including fast-reactions and linear diffusion is fundamental to the understanding of the hidden processes that cross-diffusion terms can describe. It is extremely interesting to identify such fast processes at the \emph{microscopic level} (in terms of time scales) and to obtain the \emph{macroscopic} limiting system in which these fast processes are naturally incorporated either in the reaction or in the nonlinear diffusion.
	In \cite{iida2006diffusion},  this approach has been exploited to justify the well-known cross-diffusion SKT model for competing species, by considering an instantaneous (fast) conversion between two subpopulations or states.  Recently,  the fast dichotomy dynamics has been applied to predator--prey interactions \cite{conforto2018reaction, desvillettes2019non,izuhara2023cross}, aggregation phenomena in biology \cite{eliavs2022aggregation, funaki2012link}, hunter--gatherers and farmers in Neolithic transition \cite{eliavs2021singular},  evolution of dietary diversity and starvation effects \cite{brocchieri2021evolution}, autotoxicity effects in plant growth dynamics \cite{GIS}, and behavioural epidemiology \cite{BS}.
	
	Althought the QSSA can be performed at a formal level,  a proper analysis of the convergence of solutions to the fast-reaction system to the limiting cross-diffusion one is challenging and needs sophisticated techniques \cite{conforto2014rigorous,desvillettes2015new,conforto2018reaction, CKJES2023, tang2023rigorous}.
	This paper is devoted to study the fast-reaction limit of the following predator--prey model incorporating competition among predators,  proposed in \cite{conforto2018reaction}. Let $\Omega\subset\mathbb R^d$, $d\in \mathbb N$, be a bounded domain with smooth boundary $\partial\Omega$, e.g. $\partial\Omega$ is of class $C^{2+\kappa}$ for some $\kappa>0$. The model writes
	\begin{align}\label{e1}
	\left\{ \begin{array}{clllll}
	\partial_t N^\varepsilon  - d_1\Delta N^\varepsilon  &=& r_0(1-\eta N^\varepsilon )N^\varepsilon  - \dfrac{\alpha  p_s^\varepsilon}{\xi p_s^\varepsilon+1} N^\varepsilon    , & x\in \Omega, \,t>0, \vspace*{0.15cm}\\
	\partial_t p_s^\varepsilon- d_2\Delta p_s^\varepsilon &=&   \dfrac{1}{\varepsilon} \left( \gamma  p_h^\varepsilon  - \dfrac{\alpha  p_s^\varepsilon}{\xi p_s^\varepsilon+1}N^\varepsilon  \right) -\mu p_s^\varepsilon + \Gamma p_h^\varepsilon , & x\in \Omega, \,t>0, \vspace*{0.15cm} \\
	\partial_t p_h^\varepsilon - d_3\Delta p_h^\varepsilon  &=&   \dfrac{1}{\varepsilon} \left(  \dfrac{\alpha  p_s^\varepsilon}{\xi p_s^\varepsilon+1}N^\varepsilon  - \gamma  p_h^\varepsilon  \right)     -\mu p_h^\varepsilon ,  & x\in \Omega, \,t>0,  
	\end{array}  
	\right.
	\end{align} 
	where $N^\varepsilon  \ge 0$ denotes the prey population density, $p_s^\varepsilon \ge0$ and $p_h^\varepsilon  \ge 0$ the population densities of searching and handling predators, respectively.  Assuming that the populations are confined on the domain, system \eqref{e1} is subjected to homogeneous Neumann boundary conditions
	\begin{equation}%\label{e2}
	\nabla N^\eps \cdot \nu =  \nabla p_s^\eps \cdot \nu = \nabla p_h^\eps\cdot \nu = 0, \quad x\in\partial\Omega,\, t>0,
	\end{equation}
	where $\nu$ is the outward normal vector to $\pa\Omega$ at point $x$, and non-negative initial data
	\begin{equation}%\label{e3}
	N^\eps(x,0) = N_{in}(x), \quad p_h^\eps(x,0) = p_{h,in}(x), \quad p_s^\eps(x,0) = p_{s,in}(x), \quad x\in\Omega.
	\end{equation}
	
	Here,  parameters $d_1, \,d_2, \,d_3$ are the diffusion coefficients of prey, searching, and handling predators, respectively.  We assume that handling predators move less far than searching predators, yielding the biologically meaningful constraint on the diffusion coefficients $d_3<d_2$.  In the prey equation, the prey-growth is modelled by a logistic term with maximum growth rate $r_0$ and carrying capacity $1/\eta$, while we have a loss term due to predation. This term involves the parameter $\alpha$ (attack rate) and parameter $\xi$, which measures the competition among predators for prey.  The mortality rate $\mu$ is assumed to be the same for both types of predators and only handling predators give rise, with birth rate $\Gamma$, to searching predators.  We also suppose that there is a conversion (or switch) between the two states: a searching predator becomes handling when encountering a prey and comes back to the searching state after a certain amount of time $1/\gamma$. Furthermore, we assume $d_1, \,d_2, \, d_3,\,r_0,\,\eta,\,\alpha,\,\gamma,\, \Gamma, \,\mu>0$\, and $\xi\ge 0$.  Note that, for $\xi=0$, we have no competition effect for prey among predators and the predation and switching terms reduce to the Lotka--Volterra type.
	Under the assumption that the searching--handling switch happens on a much faster time scale than the reproduction and mortality processes,  we have that $0<\varepsilon\ll 1$.
	
	\medskip
	
	As $\eps\to 0$, in some suitable sense, we formally expect that the switching dynamics reaches a \emph{quasi-steady-state}, namely
	\begin{align}
	\gamma  p_h^\varepsilon  - \dfrac{\alpha  p_s^\varepsilon}{\xi p_s^\varepsilon+1}N^\varepsilon \to 0, %\label{EpCriManiPPM}
	\end{align}
	as well as $p_s^\varepsilon \to p_s$, $p_h^\varepsilon \to p_h$,   $N^\varepsilon \to N$ which satisfy
	\begin{align}
	\gamma  p_h   - \dfrac{\alpha  p_s }{\xi p_s +1}N  = 0.  \label{CriManiPPM}
	\end{align}
	Denoting with $P$ the total predator population, we have that $P:=p_s+p_h$. Together with equation \eqref{CriManiPPM}, the component $p_h$ must solve the algebraic equation
	\begin{align}
	\gamma \xi p_h^2 - (\alpha N+\gamma +\gamma  \xi P) p_h + \alpha N P =0.  \label{CriManiPPM2}
	\end{align}
	\medskip
	We treat the case $\xi=0$ (no competition among predators) and the case $\xi>0$ (with competition effect) separately.
	For the case $\xi=0$, by solving equation \eqref{CriManiPPM2}, we obtain
	\begin{equation*}
	p_h = \frac{\alpha NP}{\alpha N + \gamma}, \quad \text{ and } \quad p_s = P - p_h = \frac{\gamma P}{\alpha N + \gamma}.
	\end{equation*}
	By adding up the second and third equations of system \eqref{e1},  we obtain an equation for the total predator population and it is easy to see the expected limiting system for $(N, P)$ is
	\begin{equation}\label{LimitSystem_xi_zero}
	\left\{
	\begin{aligned}
	&\partial_t N - d_1\Delta N = r_0(1-\eta N)N - \frac{\gamma \alpha NP}{\alpha N + \gamma}, && x\in\Omega, t>0\\
	&\partial_t P - d_2\Delta P = (d_3-d_2)\Delta \left(\frac{\alpha N P}{\alpha N + \gamma}\right) - \mu P + \Gamma \frac{\alpha N P}{\alpha N + \gamma}, && x\in\Omega, t>0,
	\end{aligned}
	\right.
	\end{equation}
	subject to homogeneous Neumann boundary conditions
	\begin{equation}\label{bc}
	\nabla N\cdot \nu = \nabla P \cdot \nu = 0, \quad x\in\partial\Omega, \; t>0,
	\end{equation}
	and initial data
	\begin{equation}\label{ic}
	N(x,0) = N_{in}(x), \quad P(x,0) = p_{h,in}(x) + p_{s,in}(x), \quad x\in \Omega.
	\end{equation}
	\medskip
	For the case $\xi>0$, due to equation \eqref{CriManiPPM}, one has $p_h<\alpha N / \gamma  \xi$. This implies
	\begin{align}
	p_h = \varphi(N,P) := \frac{A - \sqrt{B^2+4\gamma ^2 \xi P} }{2\xi \gamma }, \quad p_s=   P - \varphi (N,P), %\label{varphiDef}
	\end{align}
	where $A=A(N,P):=\alpha N+\gamma +\gamma  \xi P$ and $B=B(N,P):=\alpha N+\gamma -\gamma  \xi P$. Therefore, the couple $(N,P)$ is expected to solve the cross-diffusion system
	\begin{equation}
	\left\{ \begin{aligned}
	&\partial_t N  -d_1\Delta N = r_0( 1-\eta N )N    - \gamma \varphi(N,P) , && x\in\Omega, t>0,\\ 
	&\partial_t P -d_2  \Delta P =  (d_3-d_2)\Delta   \varphi(N,P)  + \Gamma  \varphi(N,P) -\mu P , &&x\in\Omega, t>0,
	\end{aligned}  \right.  \label{LimitSystem_xi_positive}
	\end{equation}
	subject boundary conditions \eqref{bc} and initial data \eqref{ic}.  From the modelling point of view, the limiting systems inherit and naturally incorporate in the cross-diffusion terms the istantaneous switch between the two predator states.  A by-product is the appearence of nonlinear functional responses in the reaction part: we obtain a (prey-dependent) Holling-type II and a (predator-dependent) Beddington--DeAngelis-like functional response, if $\xi=0$ and $\xi>0$, respectively.   A sketch of the compartments in the fast-reaction systems and in the limiting system is shown in Figure \ref{fig:switch}.
	
	\begin{figure}[h!]
	\centering
	\begin{tikzpicture}[even odd rule, line width=1pt, x=40pt,y=40pt,scale=0.5]
	\draw[fill=gray!10] (0,0) rectangle (2,1); %rettangolo predatori
	\draw[fill=gray!10] (0.5,2) rectangle (1.5,3);%rettangolo prede
	\draw[dotted] (1,0) -- (1,1);%separatrice predatori
        %\draw[dotted] (1,2) -- (1,3);%separatrice prede
	\draw[-, gray] (0.7,0.9) -- (0.7,2.1);%repulsione
	\draw[-,magenta] (0,0.5) -- (-0.3,0.5) -- (-0.3, 2.5) -- (0.5,2.5);%predazione
	\draw[->,thin] (0.3,0) -- (0.3,-0.3);
	\draw[->,thin] (1.7,0) -- (1.7,-0.3);
	\draw[->,thin] (1.3,0) -- (1.3,-0.15) -- (0.7,-0.15) -- (0.7,0);
	\draw node at (0.4,0.5)[] {$p^\varepsilon_s$};
	\draw node at (1.6,0.5)[] {$p^\varepsilon_h$};
	\draw node at (1,2.5)[] {$N^\varepsilon$};

	\draw[->,thin,magenta] (1.3,0.6) to[out=150,in=30] (0.7,0.6);
	\draw[->,thin,magenta] (0.7,0.4) to[out=-30,in=-150] (1.3,0.4);
	
	%\draw[->,thin] (1.3,2.6) to[out=150,in=30] (0.7,2.6);
	%\draw[->,thin] (0.7,2.4) to[out=-30,in=-150] (1.3,2.4);

    \draw[fill=gray!10] (4.5,2) rectangle (5.5,3);%rettangolo prede
    \draw[fill=gray!10] (4.5,0) rectangle (5.5,1);%rettangolo prede
%    \draw[dotted] (1,2) -- (1,3);%separatrice predatori
    \draw[-, gray, ultra thick] (5,0.9) -- (5,2.1);%repulsione
    \draw[->,thin] (5,0) -- (5,-0.3);
    
        \draw node at (3,1.5)[] {\normalsize{$\varepsilon \to 0$}};
    \draw node at (5,2.5)[] {\normalsize{$N$}};
    \draw node at (5,0.5)[] {\normalsize{$P$}};
    
%    \draw[->,thin] (1.3,2.6) to[out=150,in=30] (0.7,2.6);
%    \draw[->,thin] (0.7,2.4) to[out=-30,in=-150] (1.3,2.4);
\end{tikzpicture}

	\caption{A sketch of the compartments in the fast-reaction systems \eqref{e1} and in the limiting systems \eqref{LimitSystem_xi_zero} and \eqref{LimitSystem_xi_positive}.}\label{fig:switch}
	\end{figure}
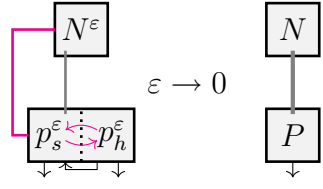
	
	\medskip

These formal limits have been proven rigorously in \cite{conforto2018reaction}. More precisely, the following results were obtained. 
\begin{theorem}[\cite{conforto2018reaction}]\label{thm1}
Assume that the initial data $N_{in}\in L^{\infty}(\Omega)$, $p_{h,in}, p_{s,in}\in L^{2+\delta}(\Omega)$ for some $\delta>0$ are non-negative. For $\eps>0$, let $(N^\eps,p_h^\eps,p_s^\eps)$ be the solution to system \eqref{e1}. Then, as $\eps\to 0$, the slow manifold converges to zero in the distributional sense up to a subsequence, i.e.
		\begin{equation}\label{conv_slow_manifold}
		\int_0^T\int_{\Omega}\psi \bra{\gamma p_h^\eps - \frac{\alpha p_s^\eps}{\xi p_s^\eps + 1}N^\eps} dxdt \to 0
		\end{equation}
		for all smooth function $\psi \in C^{2,1}(\Omega\times(0,T))$ such that $\nabla \psi \cdot \nu = 0$ on $\pa\Omega\times(0,T)$. Moreover, we have the following convergence results.
		\begin{enumerate}
			\item[(a)] If $\xi>0$,  then for $\eps\to 0$, up to a subsequence,
			\begin{equation*}
			N^\eps \to N \text{ strongly in } L^{q}(\Omega\times(0,T)), \quad p_h^\eps, p_s^{\eps} \to p_h, p_s \text{ strongly in } L^{2+\delta}(\Omega\times(0,T))
			\end{equation*}
			for all $q\in [1,\infty)$ and some $\delta>0$, and $(N,P:=p_h+p_s)$ is a very weak solution (see Definition \ref{WeakSolDef} in Section \ref{sec:3}) to the limiting system \eqref{LimitSystem_xi_positive}.
			\item[(b)] If $\xi = 0$, for $\eps\to 0$, then up to a subsequence,
			\begin{equation*}
			N^\eps \to N \text{ strongly in } L^{q}(\Omega\times(0,T)), \quad p_h^\eps, p_s^\eps \rightharpoonup p_h, p_s \text{ weakly in } L^{2+\delta}(\Omega\times(0,T)),
			\end{equation*}
			for all $q\in [1,\infty)$ and some $\delta>0$, and $(N,P:= p_h + p_s)$ is a very weak solution to the limiting system \eqref{LimitSystem_xi_zero}.
		\end{enumerate}
	\end{theorem}
	\begin{remark}\hfill
		\begin{itemize}
			\item[-] If $\xi>0$, the solution $(N,P)$ to the limiting system \eqref{LimitSystem_xi_positive} enjoys higher regularity $N \in W^{1,p}((0,T);L^p(\Omega))\cap L^p((0,T);W^{2,p}(\Omega))$ for all $T>0$, $p\ge 1$, and $p_h\in L^2(0,T;H^1(\Omega))$, $p_s\in L^1((0,T);W^{1,1}(\Omega))$.
			\item[-] If $\xi = 0$, one can show that $N\in W^{1,2+\delta}(0,T;L^{2+\delta}(\Omega))\cap L^{2+\delta}(0,T;W^{2,2+\delta}(\Omega))$ for some $\delta>0$. Moreover, if, in addition, we restrict to $d\in \{1,2\}$ and initial data $N_{in}, P_{in} = p_{h,in}+p_{s,in} \in C^{\kappa}(\bar \Omega)$, then we obtain
			\begin{equation*}
			N, P, \na N, \na P \in C^{\kappa}(\bar \Omega \times [0,T])
			\end{equation*}
			and
			\begin{equation*}
			\pa_t N, \,\pa_t P, \,\pa_{x_ix_j}N, \,\pa_{x_ix_j}P \in L^p(\Omega\times(0,T)), \quad \forall T>0,\; \forall p\in [1,\infty), \; i,j=1,\ldots, d.
			\end{equation*}
		\end{itemize}	
	\end{remark}
	
	\medskip
In the case $\xi>0$, thanks to the strong convergences in Theorem \ref{thm1} (a), we can in fact obtain the strong convergence of the slow manifold, up to a subsequence,
	\begin{equation*}
	\lim_{\eps \to 0}\left\|\frac{\alpha p_s^{\eps}}{\xi p_s^{\eps} + 1}N^\eps - \gamma p_h^\eps \right\|_{L^2(\Omega\times(0,T))} = 0
	\end{equation*}
instead of just a distributional convergence in \eqref{conv_slow_manifold}. Similarly, this strong convergence can be shown in the case $\xi = 0$ and low dimension $d\in \{1,2\}$, see \cite[Eq. (55)]{conforto2018reaction}. This improved convergence is much more subtle in the case $\xi =0$ and $d\ge 3$, since we have no strong convergence of $p_h^\eps$ and $p_s^{\eps}$. 

In this paper, by using a {\it modified energy method}, we show the strong convergence to zero of slow manifold in all dimensions without additional assumptions. This is inspired by the recent work \cite{tang2023rigorous}, where this technique was exploited to prove rigorous derivation of Michaelis--Menten kinetics for reaction--diffusion systems modelling enzyme reactions. In addition, we prove improved convergence results with respect to Theorem \ref{thm1}. More precisely, we obtain the following results.
	
	\begin{theorem}\label{thm2}
		Assume that the initial data $N_{in}\in W^{2,q}(\Omega)$ for some $q > d+2$, $p_{h,in}, p_{s,in}\in L^{2+\delta}(\Omega)$, for some $\delta>0$, are non-negative. For $\eps>0$, let $(N^\eps,p_h^\eps,p_s^\eps)$ be the solution to system \eqref{e1}.
		\begin{itemize}
			\item[(a)] If $\xi>0$, we have
			\begin{equation}\label{slow_conv1}
			\left\| \frac{\alpha p_s^\eps}{\xi p_s^\eps + 1}N^\eps - \gamma p_h^\eps\right\|_{L^2(\Omega\times(0,T))} = O(\eps^{1/2}), \quad \text{ as } \eps \to 0.
			\end{equation}
			Moreover, up to a subsequence as $\eps \to 0$,
			\begin{equation}\label{improved1}
			N^\eps, p_h^\eps, p_s^\eps \to N, p_h, p_s \quad \text{strongly in}\; L^q(\Omega\times(0,T))
			\end{equation}
			for all $q\in [1,\infty)$, and $(N, P:= p_h+p_s)$ is a weak solution of the limiting system \eqref{LimitSystem_xi_positive}.
			
			\item[(b)] If $\xi = 0$, we have
			\begin{equation}\label{slow_conv2}
			\|\alpha p_s^\eps N^\eps - \gamma p_h^\eps\|_{L^{4/3}(\Omega\times(0,T))} = O(\eps^{1/6}) \quad \text{ as } \quad \eps \to 0.
			\end{equation}
			Moreover, up to a subsequence as $\eps\to 0$,
			\begin{equation*}
			N^\eps \to N \quad\text{strongly in}\; L^q(\Omega\times(0,T))
			\end{equation*}
			for all $q\in [1,\infty)$. If the diffusion coefficients $d_2, d_3$ satisfies the additional constaint
			\begin{equation}\label{quasi-uniform-diffusion}
			\frac{d_2-d_3}{d_2+d_3} < \frac{1}{C^{\mathsf{MR}}_{q_0'}} \quad \text{ for some } \quad q_0' < 5/4,
			\end{equation}
			where $C^{\mathsf{MR}}_{q_0'}$ is the constant in the maximal regularity result (see Lemma \ref{LambertonLemma} in the Appendix), then we also have that
			\begin{equation}\label{improved2}
			p_h^\eps, p_s^\eps \to p_h, p_s \quad \text{strongly in}\; L^q(\Omega\times(0,T))
			\end{equation}
			up to a subsequence as $\eps\to 0$, for all $q\in [1,q_0)$ with $q_0 = q_0'/(q_0' - 1)$. The limit functions $(N,P:= p_h+p_s)$ is a very weak solution of the limiting system \eqref{LimitSystem_xi_zero}.
			
			%		\textcolor{red}{Bao: Should the condition \eqref{quasi-uniform-diffusion} be such that in low dimensions $d\in \{1,2\}$, it is always satisfied? We can always get strong convergence for $d\in \{1,2\}$, am I right?}
			%		
			%		\textcolor{magenta}{Bao: Apparently, for $d\in \{1,2\}$, one can prove a lower bound for $N^\eps$ by considering the equation of $\ln N^\eps$ (see \cite[Eq. (52)]{conforto2018reaction}). This is something we didn't utilize in our MM kinetics paper, and therefore we didn't have strong convergence for enzyme and complex even for $d\in \{1,2\}$.}
		\end{itemize}
	\end{theorem}
	\begin{remark}\hfill
		\begin{itemize}
		\item[-] The convergence of slow manifolds in \eqref{slow_conv1} and \eqref{slow_conv2} can be obtained in a higher $L^p$-norm with the cost of having lower convergence rate as $\eps \to 0$.
		\end{itemize}
	\end{remark}
	
Theorem \ref{thm2} improves the results of Theorem \ref{thm1} by showing the strong convergence of the slow manifold \eqref{slow_conv2} in the case $\xi = 0$ \emph{in all dimensions}, the strong convergence of $p_h^\eps, p_s^\eps$ in \eqref{improved1} for all $q\in [1,\infty)$, and the conditional strong convergence of $p_h^\eps, p_s^\eps$ in \eqref{improved2} under \eqref{quasi-uniform-diffusion}.
The core idea of our technique to prove Theorem \ref{thm2} is to consider a \textit{modified energy function}
	\begin{equation}\label{ModiEnerFunc}
	H^{\eps}(t):= \int_{\Omega}(N^\eps(x,t)+\beta(\eps))^{p-1}\bra{\int_{0}^{p_s^\eps(x,t)}\bra{\frac{\alpha r}{\xi r + 1}}^{p-1}dr}dx + \frac{1}{p}\int_{\Omega}\gamma^{p-1}(p_h^\eps)^pdx
	\end{equation}
	for $p>1$, where the function $\beta(\eps)\geq 0$ satisfies $\lim_{\eps\to 0}\beta(\eps) = 0$.  When $\xi = 0$, by taking $\beta = 0$ and $p=2$,  the energy function \eqref{ModiEnerFunc} reduces, up to a multiplication factor,  to the function used in \cite[Page 15]{conforto2018reaction}. The main advantage of having the modification $N^\eps(x,t)+\beta(\eps)$ is that it circumvents the problem of lacking a lower bound of $N^{\eps}(x,t)$ in dimensions $d\ge 3$. We also remark that this technique is suitable to treat both cases $\xi>0$ and $\xi=0$,  corresponding to Lotka--Volterra and Holling-type II terms in the fast-reaction systems, respectively, and therefore it consitutes, to some extent, a \textit{unified} approach. 
	
	\medskip
	
The paper is organised as follows.  In Section \ref{EnergyFunctional}, we prove a priori estimates using the modified energy functional $H^\varepsilon(\varepsilon)$ and the $\varepsilon$-uniform for the solution., while Section \ref{sec:3} is devoted to improved convergence results in both cases $\xi>0$ and $\xi=0$.  In the Appendix \ref{appendix},  two useful Lemmas are reported for the reader's convenience.
	
	\medskip
	
{\bf Notation:} We will use the following notation throughout the paper:
\begin{itemize}
\item[-] For fixed $T>0$ but arbitrary, and $q\in [1,\infty]$, we write 
	\begin{equation*}
	Q_T:= \Omega\times(0,T)\quad \textnormal{and} \quad L^q_{x,t}:= L^q(Q_T).
	\end{equation*}
\item[-] We write $u \in L^{q\pm}_{x,t}$ if there is some $\kappa>0$ such that $u\in L^{q\pm \kappa}_{x,t}$.
\item[-] We indicate with $C$ the constants depending on the data of the problem.
\item[-] For two quantities $\textbf{X}$ and $\textbf{Y}$, we write $\textbf{X}\lesssim \textbf{Y}$ if there exists a constant $C$ depending on the data of the problem but {\it independent of $\eps>0$} such that $\textbf{X}\le C\textbf{Y}$.
\item[-] For integrals on $\Omega$ or on $Q_T$, we suppress the differentials $dx$ and $dxdt$ respectively, for the sake of presentation.
	\end{itemize}

%%%%%%%%%%%%%%%%%%%%%%%%%%%%%%%%%%%%%%%%%%%%%%%%%%%%%%%5
\section{Energy and $\varepsilon$-uniform estimates} \label{EnergyFunctional}

In order to rigorously prove the convergence to the slow manifold as well as the convergence to the limiting system given in Theorem \ref{thm2}, we first need a priori estimates via the modified  energy functional $H^\varepsilon(t)$ defined by \eqref{ModiEnerFunc},  provided by Lemma \ref{LemmaEnergy}, and then $\varepsilon$-uniform estimates for the solution,  achieved in Lemma \ref{SecondResultCinziaPaper}.

\begin{lemma} \label{LemmaEnergy} 
Let $p>1$ and $(N^\varepsilon, p_s^\varepsilon,p_h^\varepsilon)$ be the solution to system \eqref{e1} for $\varepsilon>0$.

\begin{itemize}

\item[-] If $\xi>0$, then 
\begin{align*}
 \iint_{Q_T} &\Big( ( p_h^\varepsilon)^{p}    +     (p_h^\varepsilon)^{p-2} |\nabla p_h^\varepsilon|^2 \Big) \nonumber\\
%%%%%%%%%%%%%%%%%%%%%%
& \hspace*{1.35cm} + \frac{1}{\varepsilon} \iint_{Q_T} \bigg( \dfrac{\alpha  p_s^\varepsilon}{\xi p_s^\varepsilon+1}N^\varepsilon  - \gamma  p_h^\varepsilon \bigg) \bigg( \bigg( \dfrac{\alpha  p_s^\varepsilon}{\xi p_s^\varepsilon+1}N^\varepsilon  \bigg)^{p-1} - (\gamma  p_h^\varepsilon)^{p-1} \bigg)      \nonumber\\
%%%%%%%%%%%%%%%%%%%%%%
&   \lesssim   H^\varepsilon(0) + \iint_{Q_T} \Big( p_s^\varepsilon (N ^\varepsilon )^{p-2} |\partial_t N^\varepsilon |     +   (N ^\varepsilon)^{p-1} p_h^\varepsilon   
+    (\xi p_s^\varepsilon+1)^{2}  (N^\varepsilon  )^{p-3} |\nabla  N^\varepsilon |^2 \Big) . 
\end{align*}

\item[-] If $\xi=0$, then
\begin{align*}
\iint_{Q_T}& \Big(   (p_h^\varepsilon)^{p-2} |\nabla p_h^\varepsilon|^2 +  (N^\varepsilon + \beta(\varepsilon))^{p-1}  (p_s^\varepsilon)^{p-2}      |\nabla p_s^\varepsilon|^2 \Big) \nn \\
& \hspace*{1.35cm} + \frac{1}{\varepsilon} \iint_{Q_T} ( \alpha  p_s^\varepsilon N^\varepsilon  - \gamma  p_h^\varepsilon  ) \Big( ( \alpha  p_s^\varepsilon (N^\varepsilon + \beta(\varepsilon)) )^{p-1} - (\gamma  p_h^\varepsilon)^{p-1} \Big)      \nonumber\\
%%%%%%%%%%%%%%%%%%%%%%
&    \lesssim   H^\varepsilon (0)   +        \iint_{Q_T}   ( N^\varepsilon + \beta(\varepsilon)  )^{p-1} (p_s^\varepsilon )^{p-1}   p_h^\varepsilon   
 \nn\\
&    \hspace*{1.35cm}    +        \iint_{Q_T} \Big( (p_s^\varepsilon)^p (N^\varepsilon + \beta(\varepsilon))^{p-2} \partial_t N^\varepsilon  +   ( p_s^\varepsilon)^{p}  (N^\varepsilon + \beta(\varepsilon))^{p-3}    |\nabla N^\varepsilon |^2 \Big) .  
\end{align*} 
\end{itemize}
In both cases, the hidden constants only depend on $\alpha,\xi,\mu,\Gamma,d_2,d_3,$ and $p$.
\end{lemma}
 
\begin{proof}  Direct computations show 
\begin{align*}
\frac{dH^\varepsilon}{dt} =\,& -  \frac{1}{\varepsilon} \int_{\Omega} \bigg( \dfrac{\alpha  p_s^\varepsilon}{\xi p_s^\varepsilon+1}N^\varepsilon  - \gamma  p_h^\varepsilon \bigg) \bigg( \bigg( \dfrac{\alpha  p_s^\varepsilon}{\xi p_s^\varepsilon+1} (N^\varepsilon + \beta(\varepsilon)) \bigg)^{p-1} - (\gamma  p_h^\varepsilon)^{p-1} \bigg)   \nonumber\\
\,&+ (p-1) \int_\Omega   \int_0^{p_s^\varepsilon(x,t)}  \bigg( \frac{\alpha r}{\xi r +1}  \bigg)^{p-1}dr\, (N^\varepsilon + \beta(\varepsilon))^{p-2} \partial_t  N^\varepsilon   \nonumber\\
\,&+  \int_\Omega    \bigg( \frac{\alpha p_s^\varepsilon }{\xi p_s^\varepsilon +1}  (N^\varepsilon + \beta(\varepsilon))  \bigg)^{p-1}   (  -\mu p_s^\varepsilon + \Gamma p_h^\varepsilon )  + \int_\Omega  (\gamma p_h^\varepsilon)^{p-1} ( - \mu p_h^\varepsilon)   \nonumber\\
\,&+ d_1 \int_\Omega    \bigg( \frac{\alpha p_s^\varepsilon }{\xi p_s^\varepsilon +1}  (N^\varepsilon + \beta(\varepsilon)) \bigg)^{p-1}    \Delta p_s^\varepsilon  + d_2   \int_\Omega  (\gamma p_h^\varepsilon)^{p-1}   \Delta p_h^\varepsilon    .
\end{align*}
Applying integration by parts gives 
\begin{align*}
 \int_\Omega    \bigg( \frac{\alpha p_s^\varepsilon }{\xi p_s^\varepsilon +1}  &(N^\varepsilon + \beta(\varepsilon)) \bigg)^{p-1}    \Delta p_s^\varepsilon     \nonumber\\
& =  - \alpha (p-1) \int_\Omega (N^\varepsilon + \beta(\varepsilon))^{p-1}  \frac{(\alpha p_s^\varepsilon)^{p-2} }{(\xi p_s^\varepsilon +1)^p}    |\nabla p_s^\varepsilon|^2 \nn\\
&\qquad   - (p-1) \int_\Omega (N^\varepsilon + \beta(\varepsilon) )^{p-2}   \bigg( \frac{\alpha p_s^\varepsilon }{\xi p_s^\varepsilon +1}  \bigg)^{p-1}  \nabla p_s^\varepsilon \nabla  N^\varepsilon   \nonumber\\
& \le - \frac{\alpha (p-1)}{2} \int_\Omega (N^\varepsilon + \beta(\varepsilon))^{p-1}  \frac{(\alpha p_s^\varepsilon)^{p-2} }{(\xi p_s^\varepsilon +1)^p}    |\nabla p_s^\varepsilon|^2 \nn\\
&\qquad +  \frac{ (p-1)}{2\alpha} \int_\Omega (    N^\varepsilon + \beta(\varepsilon) )^{p-3}  \frac{(\alpha p_s^\varepsilon)^{p} }{ (\xi p_s^\varepsilon +1)^{p-2} }  |\nabla N ^\varepsilon |^2   .
\end{align*}
Therefore,
\begin{align*}
\frac{dH^\varepsilon}{dt} + \,& \frac{1}{\varepsilon} \int_{\Omega} \bigg( \dfrac{\alpha  p_s^\varepsilon}{\xi p_s^\varepsilon+1}N^\varepsilon  - \gamma  p_h^\varepsilon \bigg) \bigg( \bigg( \dfrac{\alpha  p_s^\varepsilon}{\xi p_s^\varepsilon+1} (N^\varepsilon + \beta(\varepsilon)) \bigg)^{p-1} - (\gamma  p_h^\varepsilon)^{p-1} \bigg)  \nonumber\\
\,&+  \int_\Omega (N^\varepsilon + \beta(\varepsilon))^{p-1}  \frac{(  p_s^\varepsilon)^{p-2}}{(\xi p_s^\varepsilon +1)^{p}}     |\nabla p_s^\varepsilon|^2   +   \int_\Omega (p_h^\varepsilon)^{p-2} |\nabla p_h^\varepsilon|^2   \nonumber\\
\,&+  \int_\Omega  (N^\varepsilon + \beta(\varepsilon) )^{p-1}     \frac{(p_s^\varepsilon)^p }{(\xi p_s^\varepsilon +1)^{p-1}}   +   \int_\Omega  ( p_h^\varepsilon)^{p}    \nonumber\\
\le \,& \int_\Omega   \int_0^{p_s^\varepsilon(x,t)}  \bigg( \frac{ r}{\xi r +1}  \bigg)^{p-1} dr\, (N^\varepsilon + \beta(\varepsilon))^{p-2} \partial_t  N^\varepsilon     \nonumber\\ 
\,&+ \int_\Omega    \bigg( \frac{ p_s^\varepsilon }{\xi p_s^\varepsilon +1}  (N^\varepsilon + \beta(\varepsilon)) \bigg)^{p-1}    p_h^\varepsilon   
+  \int_\Omega (    N^\varepsilon + \beta(\varepsilon) )^{p-3}    \frac{( p_s^\varepsilon)^{p} }{(\xi p_s^\varepsilon +1)^{p-2}}    |\nabla N^\varepsilon |^2   , 
\end{align*}
which directly show the desired estimates, where we choose $\beta(\varepsilon)=0$ in the case $\xi>0$.
\end{proof}

Thanks to  the proof of Theorems 1.1 and 1.2 \cite{conforto2018reaction} and the heat regularisation \cite[Lemma 2.5]{tang2023rigorous}, we directly obtain the following $\varepsilon$-uniform estimates.

\begin{lemma} \label{SecondResultCinziaPaper} Assume that the initial data satisfy the assumptions of Theorem \ref{thm2}. Let $(N^\varepsilon,p_s^\varepsilon,p_h^\varepsilon)$ be the solution to \eqref{e1} for $\varepsilon>0$.   
\begin{itemize}
\item[a)] If $\xi>0$, then for any $q\in [1,\infty)$
\begin{align*}
& \hspace*{0.4cm} \sup_{\varepsilon>0} \left( \|p_s^\varepsilon\|_{L^{2}_tH^1_{x}} + \| p_h^\varepsilon\|_{L^{2}_tH^1_{x}} +   \|N^\varepsilon\|_{L^\infty_{t}W^{1,\infty}_x}   + \|\partial_t N^\varepsilon\|_{L^q_{x,t}} + \|\Delta N^\varepsilon\|_{L^q_{x,t}} \right) \le C .
\end{align*} 
  Moreover, 
there exists a constant $ m>0$ independent of $\varepsilon$ such that 
\begin{align*}
 \esssup_{(x,t)\in Q_T} |N^\varepsilon(x,t)|  \ge m .
\end{align*}

\item[b)] If $\xi=0$, then 
\begin{align*}
& \hspace*{0.4cm} \sup_{\varepsilon>0} \left( \|p_s^\varepsilon\|_{L^{2+}_{x,t}}   + \|p_h^\varepsilon\|_{L^{2+}_{x,t}}   +  \|N^\varepsilon\|_{L^\infty_{x,t}}  + \|\nabla N^\varepsilon\|_{L^{4+}_{x,t}}  + \|\partial_t N^\varepsilon\|_{L^{2+}_{x,t}} + \|\Delta N^\varepsilon\|_{L^{2+}_{x,t}} \right) \le C .
\end{align*}
\end{itemize}
\end{lemma}
 
\section{Improved convergence}\label{sec:3}

This section is devoted to the improved convergence results provided by Theorem \ref{thm2} in both cases $\xi>0$ and $\xi=0$, where the latter is more difficult. For the sake of convenience, we will not recall in the statements the assumption on initial data in Theorem \ref{thm2},  which are $N_{in}\in W^{2,q}(\Omega)$ for  $q>N+2$, and $p_{h,in}, p_{s,in}\in L^{2+\delta}(\Omega)$  for  $\delta>0$, as well as we always let $(N^\varepsilon, p_s^\varepsilon,p_h^\varepsilon)$ be the  solution to (\ref{e1}) for $\varepsilon>0$. 

\medskip

We give a definition of weak solutions to the limiting system.

\begin{definition} \label{WeakSolDef} A nonnegative couple $(N,P)$ is called a (nonnegative) weak solution to the limiting system if 
$$N\in L^\infty(Q_T), \quad N,P \in C([0,T];L^2(\Omega)) \cap L^2(0,T;H^1(\Omega)), \quad \partial_t N, \partial_tP \in L^2(0,T;(H^1(\Omega))'),$$  
it satisfies the initial condition
\begin{equation}%\label{iclimit}
	N(x,0) = N_{in}(x), \quad P(x,0) = p_{h,in}(x) + p_{s,in}(x), \quad\text{a.e.}\quad x\in\Omega.
\end{equation}
and for all $\phi \in L^2( 0,T;H^1(\Omega))$, $\psi\in L^2(0,T;H^2(\Omega))$, $\partial_n \psi =0$ on $\partial\Omega \times (0,T)$,  

\begin{itemize}
\item[-] if $\xi>0$,  it satisfies
\begin{align}
  \begin{array}{lcl}
\displaystyle \iint_{Q_T} ( \partial_t N \phi  + d_1  \nabla N \cdot \nabla \phi ) = \displaystyle \iint_{Q_T} ( r_0( 1-\eta N )N    - \gamma \varphi(N,P) ) \phi, \vspace*{0.15cm} \nonumber \\ 
\displaystyle \iint_{Q_T} (\partial_t P \psi + d_2  \nabla P \cdot \nabla \psi ) = \displaystyle     (d_3-d_2) \iint_{Q_T} \varphi(N,P) \Delta \psi + \iint_{Q_T} ( \Gamma  \varphi(N,P)\psi -\mu P \psi ),
\end{array}  
\end{align}

\item[-] if $\xi=0$,  it satisfies
\begin{align}
  \begin{array}{lcl}
\displaystyle \iint_{Q_T} ( \partial_t N \phi  + d_1  \nabla N \cdot \nabla \phi ) = \displaystyle \iint_{Q_T} \bigg( r_0(1-\eta N)N - \frac{\gamma NP}{\alpha N + \gamma} \bigg) \phi, \vspace*{0.15cm} \nonumber \\ 
\displaystyle \iint_{Q_T} (\partial_t P \psi + d_2  \nabla P \cdot \nabla \psi )  =  \displaystyle   (d_3-d_2) \iint_{Q_T} \frac{\alpha N P}{\alpha N + \gamma} \Delta\psi + \iint_{Q_T} \bigg( \Gamma \frac{\alpha N P}{\alpha N + \gamma}\psi - \mu P \psi \bigg).
\end{array}  
\end{align}
\end{itemize}
\end{definition}

\subsection{The case $\xi>0$}

We will prove part (a) of Theorem \ref{thm2}.  The proof is based on Lemma  \ref{CriManLem}, which provides the strong convergence to the slow manifold,  Lemma \ref{LemmaRegularityPsPh}, giving the compactness of the solution sequence, and the proof of Theorem 1.2 in \cite{conforto2018reaction}.
 
\begin{lemma}% \label{CriManLem}
For any $2\le p<\infty$,
\begin{align}
\bigg\| \dfrac{\alpha  p_s^\varepsilon}{\xi p_s^\varepsilon+1}N^\varepsilon  - \gamma  p_h^\varepsilon \bigg\| _{L^p_{x,t}}  \le C \varepsilon^{1/p}. 
\end{align} 
\end{lemma}
 
\begin{proof} By the energy estimate in Lemma \ref{LemmaEnergy} and $\varepsilon$-uniform estimates in  Lemma \ref{SecondResultCinziaPaper},  we have   
\begin{align*}
& \frac{1}{\varepsilon} \iint_{Q_T} \bigg| \dfrac{\alpha  p_s^\varepsilon}{\xi p_s^\varepsilon+1}N^\varepsilon  - \gamma  p_h^\varepsilon \bigg| ^{p} \nonumber\\
& \hspace*{1cm} \le \frac{1}{\varepsilon} \iint_{Q_T} \bigg( \dfrac{\alpha  p_s^\varepsilon}{\xi p_s^\varepsilon+1}N^\varepsilon  - \gamma  p_h^\varepsilon \bigg) \bigg( \bigg( \dfrac{\alpha  p_s^\varepsilon}{\xi p_s^\varepsilon+1}N^\varepsilon  \bigg)^{p-1} - (\gamma  p_h^\varepsilon)^{p-1} \bigg) \nonumber\\
%%%%%%%%%%%%%%%%%%
& \hspace*{1cm} \lesssim   H^\varepsilon(0) +  \iint_{Q_T} \Big( p_s^\varepsilon (N ^\varepsilon )^{p-2} |\partial_t N^\varepsilon |     +   (N ^\varepsilon)^{p-1} p_h^\varepsilon   
+    (\xi p_s^\varepsilon+1)^{2}  (N^\varepsilon  )^{p-3} |\nabla  N^\varepsilon |^2 \Big) \nonumber\\
%%%%%%%%%%%%%%%%%%%
& \hspace*{1cm} \lesssim  H^\varepsilon(0) +   \|p_s^\varepsilon\|_{L^2_{x,t}} \|\partial_t N^\varepsilon \|_{L^2_{x,t}}  +  \| N ^\varepsilon\|_{L^\infty_{x,t}}^{p-1} \| p_h^\varepsilon \|_{L^1_{x,t}}   
+   \|\xi p_s^\varepsilon+1\|_{L^2_{x,t}}^{2}   \|\nabla  N^\varepsilon \|_{L^\infty_{x,t}}^2  ,
\end{align*}
where we note that
\begin{align*}
(N ^\varepsilon )^{p-2} \le \| N ^\varepsilon\|_{L^\infty_{x,t}}^{p-2}, \quad \text{and} \quad  (N^\varepsilon  )^{p-3} \le \max\{ \| N^\varepsilon \|_{L^\infty_{x,t}}^{p-3};\, m^{p-3}  \}. 
\end{align*}
Finally, we note that the term $H^\varepsilon(0)$ is bounded uniformly in $\varepsilon >0$. 
\end{proof}

By combining Lemmas \ref{LemmaEnergy} and \ref{SecondResultCinziaPaper}, $\varepsilon$-uniform boundedness of $\{p_s^\varepsilon\}$ and $\{p_h^\varepsilon\}$ will be improved significantly in   $L^p(Q_T) \cap L^2(0,T;H^1(\Omega))$ for any $1\le p<\infty$.       

\begin{lemma} \label{LemmaRegularityPsPh}
For any $1<p<\infty$,   
\begin{align*}
& \sup_{\varepsilon>0} \Big(  \|p_s^\varepsilon\|_{L^{p}(Q_T)} + \|p_h^\varepsilon\|_{L^{p} (Q_T)}  +  \|\nabla p_s^\varepsilon\|_{L^2(Q_T)} + \|\nabla  p_h^\varepsilon\|_{L^2(Q_T)} \Big) \le C .
\end{align*}
\end{lemma}

\begin{proof}
By applying Lemma \ref{LemmaEnergy}, we have 
\begin{align}   
\iint_{Q_T} & \Big( ( p_h^\varepsilon)^{p}    +     (p_h^\varepsilon)^{p-2} |\nabla p_h^\varepsilon|^2 \Big) \nonumber\\ &\lesssim   H^\varepsilon(0)+  \iint_{Q_T} \Big( p_s^\varepsilon (N^\varepsilon )^{p-2} |\partial_t N^\varepsilon |     +   (N^\varepsilon)^{p-1} p_h^\varepsilon   
+    (\xi p_s^\varepsilon+1)^{2}  (N^\varepsilon  )^{p-3} |\nabla  N^\varepsilon |^2 \Big), \label{Aaaaa}     
\end{align} 
where the right hand side exists finite for all $1<p <\infty$
due to Lemma \ref{SecondResultCinziaPaper}. Hence,  $\{p_s^\varepsilon\}$ is $\varepsilon$-uniformly bounded in $L^p(Q_T)$. 
 In addition, by adding the equations of $p_s^\varepsilon$ and $p_h^\varepsilon$ in  (\ref{e1}),  
\begin{align}
\partial_t (p_s^\varepsilon + p_h^\varepsilon) - \Delta (d_2 p_s^\varepsilon + d_3 p_h^\varepsilon) \le \Gamma p_h^\varepsilon, \label{ProofLemmaRegularityPsPh2}
\end{align}
Thanks to the duality argument \cite[Lemma 3.4]{pierre2010global}, it follows    
\begin{align*}
\|p_s^\varepsilon\|_{L^p(Q_T)} \le C(\|p_h^\varepsilon\|_{L^p(Q_T)}), 
\end{align*}
i.e. the $\varepsilon$-uniform estimate in  $L^p(Q_T)$ has been passed from $\{p_h^\varepsilon\}$ to $\{p_s^\varepsilon\}$. 

\medskip

Letting $p=2$ in \eqref{Aaaaa} gives us an  $\varepsilon$-uniform boundedness of $\{\nabla p_h^\varepsilon\}$ in $L^2(Q_T)$. To pass  the   $\varepsilon$-uniform estimate in $L^2(Q_T)$ from $\{\nabla p_h^\varepsilon\}$ to $\{\nabla p_s^\varepsilon\}$,
we can multiply two sides of \eqref{ProofLemmaRegularityPsPh2} by $p_s^\varepsilon + p_h^\varepsilon$,  to get
\begin{align*}
\iint_{Q_T} \nabla (d_2 p_s^\varepsilon + d_3 p_h^\varepsilon) \nabla ( p_s^\varepsilon + p_h^\varepsilon) \le C ,
\end{align*}
which  implies an  $\varepsilon$-uniform boundedness of $\{\nabla p_h^\varepsilon\}$ in $L^2(Q_T)$ due to the Young's inequality.
\end{proof}

\subsection{The case $\xi=0$}

We now focus on part (b) of Theorem \ref{thm2}. The main difficulty is caused by lacking an $\varepsilon$-uniformly lower bound for the prey  density $N^\varepsilon$ except in low dimensions $d\in\{1,2\}$.% and therefore we have no $\varepsilon$-uniformly lower bound for $\{(N^\varepsilon)^{p-3}\}$ as  $p<3$, which we had in the case $\xi>0$.   
The best-known $\varepsilon$-uniform boundedness of both $\{p_s^\varepsilon\}, \{p_h^\varepsilon\}$ are just in  $L^{2+}(Q_T)$ \cite{canizo2014improved}, which is not enough to imply a $\varepsilon$-uniform bound for the gradient $\{\nabla N^\varepsilon\}$ as well as to deal with the right hand side of the energy estimate in Lemma \ref{LemmaEnergy} when $\xi=0$. 

\medskip

We first show in Lemma \ref{lemma5} that the strong convergence of the slow manifold can be proven in all dimensions by using a modified energy, i.e.~choosing $\beta(\varepsilon)$ suitably in Lemma \ref{LemmaEnergy}. The strong convergence of $p_s^\varepsilon$ and $p_h^\varepsilon$ can be proven under under the condition  \eqref{quasi-uniform-diffusion}, which  we will show that it is enough to obtain strong convergence, having a lower bound for $\{N^\varepsilon\}$ and an $L^\infty(Q_T)$-bound for $\{\nabla N^\varepsilon\}$ (Lemma \ref{RegularityCaseXi0}).  Strong convergence of $p_s^\varepsilon$ and $p_h^\varepsilon$ with $d\geq 3$ without condition  \eqref{quasi-uniform-diffusion}, remains an interesting open problem.
%The convergence in \cite{conforto2018reaction} can be improved under the condition  \eqref{quasi-uniform-diffusion}, which  we will show that it is barely enough to obtain strong convergence to the limiting system although it does not ensure a lower bound for $\{N^\varepsilon\}$ and an $L^\infty(Q_T)$-bound for $\{\nabla N^\varepsilon\}$. We note that convergence to the slow manifold is unconditional, i.e., without making use of the condition \eqref{quasi-uniform-diffusion}.  

\begin{lemma}\label{lemma5} It holds   
\begin{align}
 \| \alpha  p_s^\varepsilon N^\varepsilon  - \gamma  p_h^\varepsilon \| _{L^{4/3}_{x,t}} \lesssim \varepsilon^{1/6}. \label{StrongManiCase2}
\end{align}
\end{lemma}

\begin{proof}
The energy estimate in Lemma \ref{LemmaEnergy} gives 
%\begin{align*}
%&    \frac{1}{\varepsilon} \iint_{Q_T} \Big( \alpha  p_s^\varepsilon (N^\varepsilon + \beta(\varepsilon)  -\beta(\varepsilon) ) - \gamma  p_h^\varepsilon  \Big) \Big( ( \alpha  p_s^\varepsilon (N^\varepsilon + \beta(\varepsilon)) )^{p-1} - (\gamma  p_h^\varepsilon)^{p-1} \Big)         \nonumber\\
%%%%%%%%%%%%%%%%%%%%%%%
%&   \lesssim   H(0)   +        \iint_{Q_T} (p_s^\varepsilon)^p (N^\varepsilon + \beta(\varepsilon))^{p-2} \partial_t N^\varepsilon \nn\\
%& \hspace*{1.45cm} + \iint_{Q_T} \Big( ( N^\varepsilon + \beta(\varepsilon)  )^{p-1} (p_s^\varepsilon )^{p-1}   p_h^\varepsilon   
%+  ( p_s^\varepsilon)^{p}  (N^\varepsilon + \beta(\varepsilon))^{p-3}    |\nabla N^\varepsilon |^2 \Big) . 
%\end{align*}
%Therefore, 
\begin{align*}
   \frac{1}{\varepsilon} \iint_{Q_T} & \Big( \alpha  p_s^\varepsilon (N^\varepsilon + \beta(\varepsilon)) - \gamma  p_h^\varepsilon  \Big) \Big( ( \alpha  p_s^\varepsilon (N^\varepsilon + \beta(\varepsilon)) )^{p-1} - (\gamma  p_h^\varepsilon)^{p-1} \Big)         \nonumber\\
%%%%%%%%%%%%%%%%%%%%%%
&  \lesssim   H(0) +   \frac{\beta(\varepsilon)}{\varepsilon} \iint_{Q_T}    (p_s^\varepsilon)^p    +        \iint_{Q_T} (p_s^\varepsilon)^p (N^\varepsilon + \beta(\varepsilon))^{p-2} \partial_t N^\varepsilon \nn\\
& \qquad + \iint_{Q_T} \Big( ( N^\varepsilon + \beta(\varepsilon)  )^{p-1}  (p_s^\varepsilon )^{p-1}   p_h^\varepsilon   
+  ( p_s^\varepsilon)^{p}  (N^\varepsilon + \beta(\varepsilon))^{p-3}    |\nabla N^\varepsilon |^2 \Big)  \nonumber\\
&  \lesssim   H(0) +   \frac{\beta(\varepsilon)}{\varepsilon} \iint_{Q_T}    (p_s^\varepsilon)^p   \nn \\
& \qquad +  \iint_{Q_T} \Big((\beta(\varepsilon))^{p-2} (p_s^\varepsilon)^p   |\partial_t N^\varepsilon|  +   (p_s^\varepsilon )^{p-1}   p_h^\varepsilon   
+ (\beta(\varepsilon))^{p-3}   ( p_s^\varepsilon)^{p}  |\nabla N^\varepsilon |^2 \Big)  \nonumber\\
&  \lesssim   H(0) +   \frac{\beta(\varepsilon)}{\varepsilon} \iint_{Q_T}    (p_s^\varepsilon)^p  \nn \\
& \qquad +   \iint_{Q_T} \Big((\beta(\varepsilon))^{p-2} (p_s^\varepsilon)^p   |\partial_t N^\varepsilon|  +   (p_s^\varepsilon )^{p-1}   p_h^\varepsilon   
+ (\beta(\varepsilon))^{p-3} ( p_s^\varepsilon)^{p}  |\nabla N^\varepsilon |^2 \Big)  ,
\end{align*}
in which we used $(N^\varepsilon +\beta(\varepsilon))^{s} \le  (\beta(\varepsilon))^s $ for  $s\in \{p-2;p-3\}$, and noted that the term $(N_\delta^\varepsilon)^{p-1}$ with $1<p\le 2$  is $\varepsilon$-uniformly bounded from above. By the Mean Value Theorem,   
\begin{align*}
& (p-1)\Big( \alpha  p_s^\varepsilon  (N^\varepsilon + \beta(\varepsilon)) - \gamma  p_h^\varepsilon  \Big)^2 \Big( \alpha  p_s^\varepsilon (N^\varepsilon + \beta(\varepsilon))  +  \gamma  p_h^\varepsilon\Big)^{p-2}  \nn\\
& \qquad \le \Big( \alpha  p_s^\varepsilon  (N^\varepsilon + \beta(\varepsilon)) - \gamma  p_h^\varepsilon  \Big) \Big( ( \alpha  p_s^\varepsilon (N^\varepsilon + \beta(\varepsilon)) )^{p-1} - (\gamma  p_h^\varepsilon)^{p-1} \Big),
\end{align*}
which yields 
\begin{align*}
&   (p-1) \iint_{Q_T}  \Big( \alpha  p_s^\varepsilon  (N^\varepsilon + \beta(\varepsilon)) - \gamma  p_h^\varepsilon  \Big)^2 \Big( \alpha  p_s^\varepsilon (N^\varepsilon + \beta(\varepsilon))  +  \gamma  p_h^\varepsilon\Big)^{p-2} \nonumber\\
%%%%%%%%%%%%%%%%%%%%%%
&  \qquad \lesssim    Q(\varepsilon) \left(H(0) +  \iint_{Q_T}    (p_s^\varepsilon)^p    +    (p_s^\varepsilon)^p   |\partial_t N^\varepsilon|  +   (p_s^\varepsilon )^{p-1}   p_h^\varepsilon   
+  ( p_s^\varepsilon)^{p}  |\nabla N^\varepsilon |^2   \right),
\end{align*}
where 
\begin{align*}
Q(\varepsilon):=  \varepsilon + \beta(\varepsilon)   + \varepsilon (\beta(\varepsilon))^{p-2} + \varepsilon +  \varepsilon (\beta(\varepsilon))^{p-3}. 
\end{align*}
 
Let $p=1^+$. Then all integrals on the right hand side are uniformly bounded due to Lemma~\ref{SecondResultCinziaPaper}. 
On the other hand, by choosing $\beta(\varepsilon)=\varepsilon^{1/(4-p)}$, we have  
\begin{align*}
Q(\varepsilon) \le  C\varepsilon^{1/(4-p)}.  
\end{align*}
Hence, 
\begin{align*}
\iint_{Q_T}  \Big( \alpha  p_s^\varepsilon  (N^\varepsilon + \beta(\varepsilon)) - \gamma  p_h^\varepsilon  \Big)^2 \Big( \alpha  p_s^\varepsilon (N^\varepsilon + \beta(\varepsilon))  +  \gamma  p_h^\varepsilon\Big)^{p-2} \le C \varepsilon^{1/(4-p)} . 
\end{align*}
Now, by the H\"older's inequality, we have
\begin{align*}
   \iint_{Q_T} | \alpha  p_s^\varepsilon  (N^\varepsilon + \beta(\varepsilon)) - \gamma  p_h^\varepsilon  |^{4/(4-p)} \le C \varepsilon^{2/(4-p)^2} ,
\end{align*}
which directly shows inequality \eqref{StrongManiCase2}  by substituting $p=1^+$ and applying the triangle inequality.
\end{proof}

\begin{lemma} \label{RegularityCaseXi0}  If $d_2$ and $d_3$ fulfill  the additional condition \eqref{quasi-uniform-diffusion} and $N_{in}\in W^{2,q_0}(\Omega)$, $q_0=q'_0/(q'_0-1)$, then  
\begin{align} 
	&   \sup_{\varepsilon>0} \left(    \| p_s^\varepsilon\|_{L^{q_0}_{x,t}\cap L^2_tH^1_x}  +   \|p_h^\varepsilon\|_{L^{q_0}_{x,t}\cap L^2_tH^1_x}    \right) \le  C , \label{RegCaseXi0N}
\end{align}
and
\begin{align} 
	&   \sup_{\varepsilon>0} \left( \|N^\varepsilon\|_{L^{\infty}_{x,t}} + \|\partial_t N^\varepsilon\|_{L^{q_0}_{x,t}} + \|\nabla N^\varepsilon\|_{L^{4+}_{x,t} \cap L^{2q_0}_{x,t}}   \right) \le  C .  \label{RegCaseXi0P} 
\end{align}
\end{lemma}

\begin{proof} Under condition \eqref{quasi-uniform-diffusion}, the improved duality estimates \cite{canizo2014improved,einav2020indirect}, reported in Lemma \ref{ImprovedDuaEst} in Appendix \ref{appendix}, gives 
\begin{align*}
\sup_{\varepsilon>0} \left( \|p_s^\varepsilon\|_{L^{q_0}_{x,t}}  +   \|p_h^\varepsilon\|_{L^{q_0}_{x,t}} \right) \lesssim \left( \|p_{s0} \|_{L^{q_0}_{x}}  +   \|p_{h0} \|_{L^{q_0}_{x}} \right)  ,
\end{align*}
which via the heat regularisation \cite[Lemma 2.5]{tang2023rigorous} and the $L^{2q_0}_{x,t}$ in \cite{conforto2018reaction} consequently implies
\begin{align*}
\sup_{\varepsilon>0} \left(     \|\partial_t N^\varepsilon\|_{L^{q_0}_{x,t}} + \|\nabla N^\varepsilon\|_{L^{4+}_{x,t} \cap L^{2q_0}_{x,t}} \right) \le  C .
\end{align*}

We will show that  $p_h^\varepsilon$ is uniformly bounded in  $L^{q_0}_{x,t}\cap L^2_tH^1_x$. Thanks to  Lemma \ref{LemmaEnergy} with $\beta(\varepsilon)=0$, for $1<p\le2$ we have
\begin{align*}
  \iint_{Q_T}   (p_h^\varepsilon)^{p-2} |\nabla p_h^\varepsilon|^2   
  &  \lesssim   H^\varepsilon(0)   +        \iint_{Q_T}     (p_s^\varepsilon )^{p-1}   p_h^\varepsilon   
+ \iint_{Q_T} ( p_s^\varepsilon)^{p}  (N^\varepsilon)^{p-3}    |\nabla N^\varepsilon |^2   \nn\\
&\qquad     +        \iint_{Q_T} \Big( (p_s^\varepsilon)^p (N^\varepsilon )^{p-2} \partial_t N^\varepsilon  - (p_s^\varepsilon)^{p-2}      (N^\varepsilon )^{p-1}   |\nabla p_s^\varepsilon|^2 \Big) \nn\\
&=: H^\varepsilon(0)+ \iint_{Q_T}     (p_s^\varepsilon )^{p-1}   p_h^\varepsilon  +I_1^\varepsilon+I_2^\varepsilon . 
\end{align*}

\medskip

\noindent \textit{Estimate $I_1^\varepsilon$:}
 By direct computations,  multiplying the equation of $N^\varepsilon$ by $(N^\varepsilon)^{-\vartheta}$ and then integrating on $Q_T$ yields that  $\{|\nabla N|^2/(N^\varepsilon)^{1+\rho}\}$ is $\varepsilon$-uniformly bounded in $L^1(Q_T)$ for any $\vartheta\in[0,1)$. Since $q_0>5$,  we can choose $1<p\le2$   such that $
 2p/(p-1) < q_0 -1.$ 
Then, by taking $\vartheta$  such that  $2-p<\vartheta<1$,  
\begin{align*}
\frac{p\vartheta+p}{\vartheta+p-2} < \frac{2p}{p-1} < q_0 -1,
\end{align*} 
which gives $(q_0-1)(\vartheta+p-2)-p(1+\vartheta) >0$ and deduces
\begin{align}
r:=\frac{q_0(1+\vartheta)}{q_0(\rho+p-2)-p(1+\vartheta)} \le \frac{q_0(1+\vartheta)}{\vartheta+p-2}. \label{BetaCase1}
\end{align} 
Note that we can choose the number $r$  strictly greater than $1$ by taking $\vartheta$  close enough to $1$. 
Now, applying the H\"older's inequality gives
\begin{align*}
& \sup_{\varepsilon>0} I_1^\varepsilon 
  = \sup_{\varepsilon>0} \iint_{Q_T} ( p_s^\varepsilon)^{p} \frac{|\nabla N^\varepsilon|^{\frac{2(3-p)}{1+\vartheta}}}{(N^\varepsilon)^{3-p}} |\nabla N^\varepsilon|^{\frac{2(\vartheta+p-2)}{1+\vartheta}} \nn\\
& \hspace*{1cm} \le  \sup_{\varepsilon>0} \left( \| p_s^\varepsilon \|_{L^{q_0}_{x,t}}^2 \bigg\| \frac{|\nabla N^\varepsilon|^{2}}{(N^\varepsilon)^{1+\vartheta}} \bigg\|_{L^1_{x,t}}^{\frac{3-p}{1+\vartheta}} \|\nabla N^\varepsilon\|^{\frac{2(\vartheta+p-2)}{1+\vartheta}}_{L^{ \frac{2(\vartheta+p-2)}{1+\vartheta}r}_{x,t}} \right)  <\infty,  
\end{align*}
where the last factor is finite due to inequality \eqref{BetaCase1} and $\|\nabla N^\varepsilon\|_{L_{x,t}^{2q_0}}\leq C$. 

\medskip

\noindent \textit{Estimate $I_2^\varepsilon$:} By using the equation for $N^\varepsilon$, it is clear that 
\begin{align}
    I_2^\varepsilon  
= & - d_1(p-2) \iint_{Q_T} (p_s^\varepsilon)^p (N^\varepsilon )^{p-3} | \nabla N^\varepsilon |^2 - d_1 p \iint_{Q_T}  (p_s^\varepsilon)^{p-1} (N^\varepsilon )^{p-2}  \nabla p_s^\varepsilon \nabla N^\varepsilon  \nn\\
 &- \iint_{Q_T} (N^\varepsilon )^{p-1}  (p_s^\varepsilon)^{p-2}      |\nabla p_s^\varepsilon|^2 + \iint_{Q_T} (N^\varepsilon )^{p-1} (p_s^\varepsilon)^p (r_0(1-\eta N^\varepsilon)  - \alpha  p_s^\varepsilon). \nn
\end{align}
Hence,
\begin{align*}
\sup_{\varepsilon>0} I_2^\varepsilon  
&\le \big(d_1(2-p) + (d_1p)^2 \big) \sup_{\varepsilon>0} I_1^\varepsilon    
+ r_0\iint_{Q_T} (N^\varepsilon )^{p-1} (p_s^\varepsilon)^p <\infty,   
\end{align*}
where we have used the estimate  
\begin{align*}
- d_1p \iint_{Q_T} (p_s^\varepsilon)^{p-1} (N^\varepsilon )^{p-2}  \nabla p_s^\varepsilon \nabla N^\varepsilon \le   (d_1p)^2 I_1^\varepsilon  +  \iint_{Q_T} (p_s^\varepsilon)^{p-2} (N^\varepsilon )^{p-1} |\nabla p_s^\varepsilon|^2   .  
\end{align*}
By taking $p=2$, we have $\{p_h^\varepsilon\}$ is $\varepsilon$-uniformly bounded in $L^{q_0}(Q_T)\cap L^2((0,T);H^1(\Omega))$. Thanks to Lemma \ref{LemmaRegularityPsPh}, we can pass the gradient estimate from $p_h^\varepsilon$ to $p_s^\varepsilon$, and therefore  $\{p_s^\varepsilon\}$ is $\varepsilon$-uniformly bounded in  $L^{q_0}_{x,t}\cap L^2_tH^1_x$.
\end{proof}

Now we are ready to prove the second part of Theorem  \ref{thm2}. 

\begin{proof}[{\bf Proof of Theorem \ref{thm2}b}]
 By the Aubin--Lions lemma, we implies from  inequality \eqref{RegCaseXi0N} that   $\{N^\varepsilon\}$ is relatively compact in $L^{\infty}(Q_T)$. Hence there is a subsequence, which will be also denoted by $\{N^\varepsilon\}$ that strongly converges to $N$ in $L^{\infty}(Q_T)$.  It reads  
\begin{align}
N^\varepsilon \to N \quad \text{in}\quad L^{\infty}(Q_T). \label{NCon}
\end{align}
On the other hand,  since $$ \partial_t P^\varepsilon = d_2 \Delta p_s^\varepsilon + d_3\Delta p_h^\varepsilon  - \mu  p_s^\varepsilon + (\Gamma-\mu) p_h^\varepsilon,$$ the sequence $\{\partial_t P^\varepsilon\}$ is uniformly bounded in $L^2(0,T;(H^1(\Omega))')$. Moreover, it follows from inequality \eqref{RegCaseXi0P} that $\{\nabla P^\varepsilon\}$ is bounded in $L^{2}(Q_T)$.  By applying the Aubin--Lions lemma,   $\{P^\varepsilon\}$ is relatively compact in $L^2(Q_T)$. Therefore, with the regularity given by inequality \eqref{RegCaseXi0N}, we have 
\begin{align}
P^\varepsilon \to P \quad \text{in}\quad L^{q_0}(Q_T). \label{PCon} 
\end{align}
up to a subsequence. 

\medskip

By Lemma \ref{RegularityCaseXi0},  $\alpha p_h^\varepsilon - \alpha p_s^\varepsilon N^\varepsilon$ is uniformly bounded in $L^{q_0}(Q_T)$. Then, by diagonalizing and up to subsequences, $
\gamma p_h^\varepsilon - \alpha p_s^\varepsilon N^\varepsilon \to 0$ in $L^{q}(Q_T)$ 
for all $q<q_0$. This ensures that
\begin{align*}
\bigg\| p_h^\varepsilon - \frac{\alpha N^\varepsilon P^\varepsilon}{\alpha N^\varepsilon + \gamma} \bigg\|_{L^{q_0-}(Q_T)}   \le \frac{1}{\gamma}  \| \gamma p_h^\varepsilon - \alpha p_s^\varepsilon N^\varepsilon  \|_{L^{q_0-}(Q_T)} \to 0.
\end{align*}
Therefore, the triangle inequality yields
\begin{align}
 p_h^\varepsilon \to  \frac{\alpha N P }{\alpha N + \gamma}     \quad \text{in} \quad L^{q_0-}(Q_T)  \label{phCon}
\end{align}
as well as 
\begin{align}
 p_s^\varepsilon \to  \frac{\gamma P }{\alpha N + \gamma}  \quad \text{in} \quad L^{q_0-}(Q_T) . \label{psCon}
\end{align}
Now by adding the equations of $p_h^\varepsilon$, $p_s^\varepsilon$, and combining the resultant with the equation of $N^\varepsilon$,  we obtain
\begin{align} 
\left\{ \begin{array}{clllll}
 \partial_t N^\varepsilon  - d_1\Delta N^\varepsilon  &=& r_0(1-\eta N^\varepsilon )N^\varepsilon  -  \alpha  p_s^\varepsilon  N^\varepsilon  , \vspace*{0.15cm}\\
    \partial_t P^\varepsilon- d_2\Delta P^\varepsilon &=& (d_3-d_2)\Delta p_h^\varepsilon    + \Gamma p_h^\varepsilon  -\mu P^\varepsilon  ,   
\end{array}  
\right.\nn
\end{align}
which, together with  \eqref{NCon}--\eqref{psCon}, gives that $(N,P)$ is a weak solution to the limiting system. 
\end{proof}

\medskip

\noindent{\textbf{Acknowledgement.}} B.Q. Tang and B.-N. Tran has received funding from FWF under the FWF project ``Quasi-steady-state approximation for PDE", number I-5213. C.~Soresina is a member of the Istituto Nazionale di Alta Matematica (INdAM), Gruppo Nazionale per la Fisica Matematica (GNFM).

 %%%%%%%%%%%%%%%%%%%%%%%%%%%%%%%%
\appendix
\section{Appendix}\label{appendix}

We report here, for the reader's convenience,  two useful Lemmas.

\begin{lemma}[{\cite[Theorem 1]{lamberton1987equations}}]  \label{LambertonLemma}
	Let $D>0$. Assume that $f\in L^p(Q_T)$, $1<p<\infty$, and let $u$ be a weak solution to problem 
\begin{equation}	\nonumber
  	\begin{cases}
		\partial_t u - D \Delta u = f, &\text{ in }  Q_T,\\
		\nabla u\cdot{\nu} = 0, &\text{ on } \pa\Omega\times(0,T),\\
		u(x,0) = 0, &\text{ in } \Omega.
	\end{cases}
	\end{equation}	
	 Then there is an optimal constant $C_{p}^{\MR}$ depending only on $p,d,\Omega$ and $D$, such that
	\begin{equation}\label{maxreg2}
		\|\Delta u\|_{L^p_{x,t}}\le C_{p}^{\MR}\|f\|_{L^p_{x,t}} , 
	\end{equation}
	where the superscript $\MR$ indicates the Maximal Regularity property.
	\end{lemma}

\begin{lemma}[Improved duality estimate, \cite{canizo2014improved,einav2020indirect}] \label{ImprovedDuaEst} Let  $T>0$,  $1<q<\infty$, $k\in \mathbb{R}$. Assume that $X,Y$ are nonnegative, smooth functions  satisfying the relation
	\begin{align}
	\left\{ \begin{array}{lllllll}
	%%%%%1
	\displaystyle \partial_t (X+Y)  &\le&  \Delta ( a X + b Y ) + k (X+Y)   & \text{in } Q_T, \vspace*{0.15cm}\\
	\nabla X\cdot \nu = \nabla Y\cdot \nu&=& 0    & \text{on } \partial \Omega \times (0,T),
	\end{array} 
	\right.  \label{CrossDiffProblem}
	\end{align}
	for some constants $a,\,b >0$.   If  
	\begin{align}
	 \frac{|a-b|}{a+b}C_{q'}^{\MR}  < 1,    \label{SmallDifferenced2d3}
	\end{align}
where $q' = q/(q-1)$ is H\"older conjugate exponent of $q$,  then  
	\begin{align}
	 \left\|X\right\|_{L^q(Q_T)} + \left\|Y\right\|_{L^q_{x,t}} 
	\lesssim \left\|  X(0) + Y(0) \right\|_{L^q_x}, 
\label{ImproveEstimate} 
	\end{align}
	where the hidden constant depends continuously on $T,a,b$.
\end{lemma}

 \newcommand{\noop}[1]{}

\end{document}